\documentclass[10pt]{article}
\usepackage{amssymb}
\usepackage{graphicx}
\usepackage{xcolor} % A package to add color.
\usepackage{tensor}
\usepackage{fullpage} % Sets all margins to 1 in.  
\usepackage{amsmath}
\usepackage{amsthm}
\usepackage{verbatim}
\usepackage{hyperref}
\usepackage{enumitem}
\setlist[enumerate]{leftmargin=1.5em}
\setlist[itemize]{leftmargin=1.5em}

\providecommand{\MR}{\relax\ifhmode\unskip\space\fi MR }
\providecommand{\href}[2]{#2}

\usepackage{xcolor}
\definecolor{purple}{rgb}{0.5, 0, 1}
\definecolor{orange}{rgb}{1,.5,0}

\providecommand{\MR}{\relax\ifhmode\unskip\space\fi MR }
\providecommand{\href}[2]{#2}

\usepackage{seqsplit} 
\definecolor{green}{rgb}{0,0.8,0} % Redefines the color green.
\newtheorem{maintheorem}{Theorem}

\newtheorem{theorem}{Theorem}[section]

\newtheorem{lemma}[theorem]{Lemma}

\theoremstyle{definition}

\theoremstyle{remark}
\newtheorem{remark}[theorem]{Remark}

\numberwithin{equation}{section}
\newcommand{\nrm}[1]{\Vert#1\Vert}

\newcommand{\nnrm}[1]{{\vert\kern-0.25ex\vert\kern-0.25ex\vert #1 
    \vert\kern-0.25ex\vert\kern-0.25ex\vert}}

\newcommand{\supp}{{\mathrm{supp}}\,}

\newcommand{\lap}{\Delta}

\newcommand{\ud}{\mathrm{d}}
\newcommand{\rd}{\partial}
\newcommand{\nb}{\nabla}

\newcommand{\alp}{\alpha}
\newcommand{\bt}{\beta}
\newcommand{\gmm}{\gamma}

\newcommand{\dlt}{\delta}

\newcommand{\tht}{\theta}

\newcommand{\bbN}{\mathbb N}

\newcommand{\bbR}{\mathbb R}

\begin{document}

\title{On well-posedness of $\alpha$-SQG equations in the half-plane}%: Title of the article
\author{In-Jee Jeong\thanks{Department of Mathematics and RIM, Seoul National University. E-mail: injee\_j@snu.ac.kr} \and Junha Kim\thanks{School of Mathematics, Korea Institute for Advanced Study. E-mail: junha02@kias.re.kr} \and Yao Yao\thanks{Department of Mathematics, National University of Singapore. E-mail: yaoyao@nus.edu.sg}} 
\date{\today}

\renewcommand{\thefootnote}{\fnsymbol{footnote}}

\renewcommand{\thefootnote}{\arabic{footnote}}

%\date{\today}%
%\dedicatory{}%
%\commby{}%
% ----------------------------------------------------------------

\maketitle

% ----------------------------------------------------------------

\begin{abstract}
    We investigate the well-posedness of $\alpha$-SQG equations in the half-plane, where $\alpha=0$ and $\alpha=1$ correspond to the 2D Euler and SQG equations respectively. For $0<\alpha \le 1/2$, we prove local well-posedness in certain weighted anisotropic H\"older spaces. We also show that such a well-posedness result is sharp: for any $0<\alpha \le 1$, we prove nonexistence of H\"older regular solutions (with the H\"older regularity depending on $\alpha$) for initial data smooth up to the boundary. 
\end{abstract}

%\tableofcontents

\section{Introduction}

\subsection{Generalized SQG equations}

In this paper, we are concerned with the Cauchy problem for the inviscid $\alpha$-surface quasi-geostrophic  ($\alpha$-SQG) equations on the right half plane $\bbR^2_+ = \{ (x_1,x_2) \in \bbR^2 ; x_1 > 0 \}$
\begin{equation}  \label{eq:SQG} \tag{$\alpha$-SQG}
\left\{
\begin{aligned} 
&\rd_t\tht + u\cdot \nb \tht = 0, \\
&u=  - \nb^\perp (-\lap)^{-1 + \frac{\alp}{2}} \tht,
\end{aligned}
\right.
\end{equation}   for $0\le \alp \le 1$. For notational simplicity, throughout this paper we shall normalize the constant in a way that the Biot--Savart law becomes \begin{equation}\label{eq:Biot-Savart-aSQG}
	\begin{split}
		u(t,x) = \int_{ \bbR^2_+ } \left[ \frac{(x-y)^\perp}{|x-y|^{2+\alp}} - \frac{(x-\tilde{y})^\perp}{|x-\tilde{y}|^{2+\alp}}   \right] \tht(t,y) \, \ud y,
	\end{split}
\end{equation} where $\tilde{y}:=(-y_1,y_2)$ for $y=(y_1,y_2)$.

In the last decade, the $\alpha$-SQG equations (either with or without dissipation) in domains with boundaries have attracted a lot of attention. The constructed solutions were either $L^p$ weak solutions, patch solutions, or solutions that vanish on the boundary. Below we summarize the previous literature on well-posedness of solutions.

\begin{itemize}
    \item Weak solutions: For the generalized SQG (or SQG) equation, when the equation is set up in a domain with a boundary, global existence of weak solution in  $L^\infty_tL^2_x$ was established in \cite{CNgu, Ngu, constantin2018inviscid}. We refer to \cite{Res, Mar, BaeGra} for the existence of weak solutions in $\bbR^2$ and \cite{BSV, IM,CKL} for the non-uniqueness of the weak solutions. Existence results of weak solutions in $\bbR^2$ can be directly applied to give weak solutions in $\bbR^2_+$ by considering solutions which are odd in one variable. 
    \item Patch solutions: For the 2D Euler equation on the half plane, global existence of $C^{1,\gamma}$ patch solutions was shown in \cite{Dep, D03, KRYZ}. For the $\alpha$-SQG equations on the half plane, local well-posedness of $H^3$-patch solutions for $\alpha \in (0, \frac{1}{12})$ was established in \cite{KYZ}, and it was shown in \cite{KRYZ} that such patch solution can form a finite time singularity for this range of $\alpha$. Later, \cite{GaPa} extended the local well-posedness and finite-time singularity results to $\alpha\in(0,\frac{1}{6})$ for $H^2$-patch solutions. The creation of a splash-like singularity was ruled out in the works \cite{JZ-axi,KL23}. We also refer to \cite{CGI} for stability of the half-plane patch stationary solution.
\end{itemize}
 
To the best of our knowledge, in the literature there has been no local well-posedness results for strong solutions of \eqref{eq:SQG} which \textit{do not} vanish on the boundary. The difficulty is mainly caused by the fact that when $\theta$ do not vanish on the boundary, even if it is smooth, the regularity of the velocity field $u$ becomes worse as we get closer to the boundary, and $u$ is only $C^{1-\alpha}$ near the boundary. A closer look reveals that the two components of $u$ have different regularity properties. This motivates us to introduce an anisotropic functional space, which turns out to be crucial for us to establish the local-wellposedness result.

\medskip

While we were finalizing the paper,  a preprint by Zlato\v{s} \cite{Zlatos} appeared where the well-posedness problem is studied in a very similar setting, and the same type of anisotropic weighted space was discovered independently. A detailed discussion can be found in Remark~\ref{rmk_singularity} and the end of Section~\ref{sec_results}.

\subsection{Main Results}\label{sec_results}
For any $0<\bt\le 1$, we introduce the space $X^{\bt} = X^{\bt}(\overline{\bbR^2_+})$, which is a subspace of $C^{\bt}(\overline{\bbR^2_+})$ with anisotropic Lipschitz regularity in space: we say $f \in X^{\bt}$ if it belongs to $C^{\bt}$, differentiable almost everywhere, and satisfies \begin{equation}\label{eq:space}
	\begin{split}
		\nrm{f}_{X^{\bt}} := \nrm{f}_{C^{\bt}} + \nrm{x_1^{1-\bt} \rd_1 f}_{L^\infty} + \nrm{ \rd_2 f }_{L^\infty} < \infty. 
	\end{split}
\end{equation} The H\"older norm is defined by \begin{equation}\label{eq:space2}
	\begin{split}
		\nrm{f}_{C^{\bt}} := \nrm{f}_{L^\infty} + \sup_{x\ne x'} \frac{|f(x)-f(x')|}{|x-x'|^{\bt}}. 
	\end{split}
\end{equation}  For simplicity, we only deal with the functions $f$ that are compactly supported in a ball of radius $O(1)$, so that we have $X^{\bt}\subset X^{\gmm}$ if $\bt\ge\gmm$. In general, one may replace the weight $x_1^{1-\bt}$ by something like $(x_1/(1+x_1))^{1-\bt}$. 

It turns out that, in the regime  $\alpha \in (0,\frac {1}{2}]$, the $\alp$-SQG equation is locally well-posed in $X^{\bt}_c$. In the following theorem we prove the local well-posedness result, and also establish a blow-up criterion: if the solution in $X^{\bt}_c$ cannot be continued past some time $T<\infty$, then $\| \partial_2 \tht(t,\cdot) \|_{L^{\infty}}$ must blow up as $t\to T$.
\begin{maintheorem}\label{thm_loc1}
	Let $\alpha \in (0,\frac {1}{2}]$ and $\beta \in [\alpha,1-\alpha]$. Then \eqref{eq:SQG} is locally well-posed in $X^{\bt}_c$: for any $\tht_{0} \in X^{\bt}_c$, there exist $T = T(  \nrm{\tht_0}_{X^{\alpha}}, |\supp \tht_0| )>0$ and a unique solution $\tht$ to \eqref{eq:SQG} in the class $\mathrm{Lip}([0,T);L^{\infty}) \cap L^\infty([0,T); X^{\bt}_c) \cap C([0,T); C^{\bt'})$ for any $0<\bt'<\bt$. Furthermore, if $T<\infty$, then $\lim_{t\to T} \| \partial_2 \tht(t,\cdot) \|_{L^{\infty}} \gtrsim (T-t)^{-\eta}$ holds for any $\eta \in (0,1)$.
\end{maintheorem}

\begin{remark}(Finite-time singularity formation in $X^\beta_c$). \label{rmk_singularity}
	For $\alpha\in(0,\frac{1}{3})$, one can follow the same argument in \cite{GaPa, KRYZ} to construct a $\theta_0\in  C^{\infty}_c(\overline{\bbR^2_+})$ where the solution leaves $X^{\bt}_c$ in finite time. To show this, let us start with the initial data $1_{D_0}$ in \cite{GaPa, KRYZ}, where $D_0$ consists of two disjoint patches symmetric about the $x_1$ axis. (In our setting they are symmetric about the $x_2$ axis, since our domain is the right half plane.) We then set our initial data $\theta_0\in  C^{\infty}_c(\overline{\bbR^2_+})$ such that $\theta_0 \equiv 1$ in $D_0$, $\theta_0 \geq 0$ in $\overline{\bbR^2_+}$, and $\theta_0 \equiv 0$ on the $x_1$-axis. Assuming a global-in-time solution in $X^\beta$, on the one hand we have $\theta(t,\cdot) \equiv 0$ on the $x_1$-axis for all time by symmetry. On the other hand, one can check that all the estimates in \cite{GaPa, KRYZ} on $u$ still hold, thus the set $\{\theta(t,\cdot)=1\}$ would touch the origin in finite time, leading to a contradiction.  
 
For $\alpha>\frac13$, significantly new ideas are needed to prove the singularity formation. In the very recent preprint by Zlato\v{s} \cite{Zlatos}, he used some more precise and delicate estimates to improve the parameter regime to cover the whole range $\alpha\in(0,\frac12]$ where the equation is locally well-posed.

Finally, we point out that the blow-up criteria in Theorem~\ref{thm_loc1} implies that at the blow-up time $T$, we must have $\lim_{t\to T} \| \partial_2 \tht(t,\cdot) \|_{L^{\infty}} = \infty$. It would be interesting to figure out what is the exact regularity of the solution at the blow-up time.  
\end{remark} 

As a consequence of Theorem \ref{thm_loc1}, if the initial data is smooth and compactly supported in $\bbR^2_+$, there is a unique local solution in $L^\infty([0,T); X^{1-\alpha}_c)$. Our second main result shows that this regularity is sharp, even for $C^\infty_c$-data. In fact, for \emph{all} $C^\infty_c$ initial data that do not vanish on the boundary, we show that the $C^\bt$-regularity of the solution is \textit{instantaneously lost} for all $\bt>1-\alp$.   
\begin{maintheorem}\label{thm:nonexist1}
	Let $\alpha \in (0,\frac {1}{2}]$ and assume $\tht_0 \in C^{\infty}_c(\overline{\bbR^2_+})$ does not vanish on the boundary. Then, the local-in-time solution $\tht$ to \eqref{eq:SQG} given by Theorem~\ref{thm_loc1} does not belong to $L^\infty([0,\delta];C^{\beta}(\overline{\bbR^2_{+}}))$ for any $\beta > 1-\alpha$ and $\delta>0$. 
\end{maintheorem}
\begin{remark}
	The condition on the initial data necessary for illposedness in Theorem \ref{thm:nonexist1} can be relaxed to $$\tht_0(x_0) \neq 0, \qquad \limsup_{x \to x_0, x \in \partial \bbR^2_{+}} \frac {|\tht_0(x_0) - \tht_0(x)|}{|x_0 - x|} > 0$$ for some $x_0 \in \partial \bbR^2_{+}$.
\end{remark}
%\begin{remark}
%	Since the solutions must satisfy $\| \partial_2\tht(t) \|_{L^{\infty}} < \infty$ on some interval $[0,\delta]$ by Theorem \ref{thm_loc1}, the smoothness of $\tht$ should break down in $x_1$ variable.
%\end{remark}

Our last main result, which deals with the case $\alp>1/2$, shows that there is \textit{nonexistence} of solutions not only in $X^{\alp}_c$ but even in $C^{\alp}$. 
\begin{maintheorem}\label{thm:nonexist2}
	Let $\alpha \in (\frac {1}{2},1]$ and assume $\tht_0 \in C^{\infty}_c(\overline{\bbR^2_+})$ does not vanish on the boundary. Then, there is no solution to \eqref{eq:SQG} with initial data $\tht_0$ belonging to $L^\infty([0,\delta];C^{\alpha}(\overline{\bbR^2_{+}}))$ for any $\delta>0$.
\end{maintheorem}

Our main results are summarized in Figure \ref{fig:wp}. Namely, for $\alpha \in (0,1]$ and $\beta \in [0,1]$, we have the following three distinct regimes:

\begin{figure}[htbp]
	\centering
	\includegraphics[scale=1]{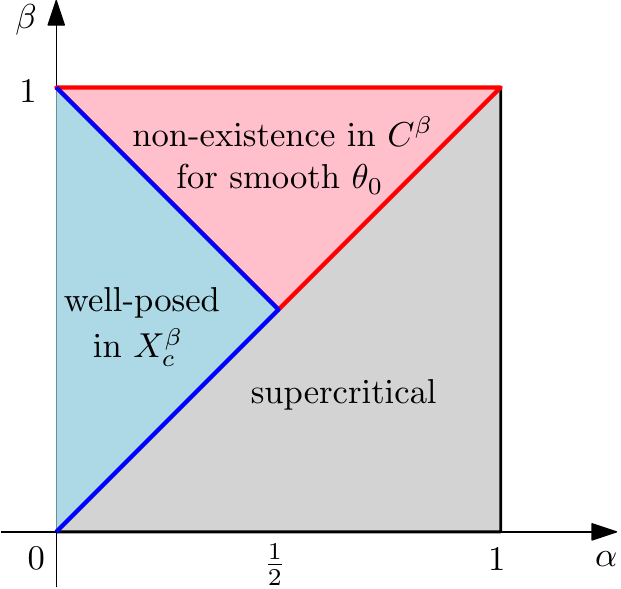}  
	\caption{Illustration of well-posedness result of $\alp$-SQG in $X^{\bt}_c$ spaces (in blue color), and non-existence results in $C^\beta$ for smooth initial data (in red color).} \label{fig:wp}
\end{figure}
\begin{itemize}
    \item The ``wellposed'' region is given by $\{0< \alp \le 1/2,\alp \le \bt \le 1-\alp\}$: note that the boundary points are included (Theorem \ref{thm_loc1}). The case $\bt=\alp$  with $\alp>0$ is especially interesting since it is known that $\alpha$-SQG equations are known to be illposed in the critical space $C^{\alp}$ (and also in $H^{1+\alpha}$) in $\mathbb{R}^2$ (\cite{JK-axi,CM,EM1}). On the other hand, our result shows that when $\alp\le 1/2$, we have well-posedness for datum which behaves like $\tht_{0}(x) \sim \mathrm{const} + x_1^{\alp}$ near the boundary $\{ x_1 = 0 \}$, which is exactly $C^{\alp}$ and not better in the scale of H\"older spaces. 
    \item On the other hand, the ``illposed'' region is given by $\{ 1/2<\bt \le1 , 1-\bt < \alp \le \bt\}$ for any given smooth and compactly supported initial data that do not vanish on the boundary (Theorems \ref{thm:nonexist1} and \ref{thm:nonexist2}). Note that the illposedness is proved not just in $X^{\bt}$ but in $C^{\bt}$. Indeed, if we consider any singular initial data $\tht_0 \in C^{\alpha}_c(\overline{\bbR^2_{+}})$ with $$\tht_0(x_0) \neq 0, \qquad \limsup_{x \to x_0, x \in \partial \bbR^2_{+}} \frac {|\tht_0(x_0) - \tht_0(x)|}{|x_0 - x|^{\alpha}} > 0$$ for some $x_0 \in \partial \bbR^2_{+}$, it is not hard to prove that \eqref{eq:SQG} are ill-posed in the critical H\"older space.
    \item Lastly, the regime $\bt<\alp$ is ``supercritical'' with respect to the scaling of the $\alp$-SQG equations; local well-posedness in $C^{\bt}$ (or even in $X^{\bt}$) is not expected, even in the whole space case $\bbR^{2}$. There have been several exciting developments in this regime: see \cite{Elling2016,CM}.
\end{itemize} 
  
 In the very recent preprint by Zlato\v{s} \cite{Zlatos}, the local well-posedness of solutions was established in the same regime of parameters as our Theorem~\ref{thm_loc1}. Compared to our work, \cite{Zlatos} obtained ill-posedness in $X^\beta$ for broader parameters, covering both our ill-posed regime and supercritical regime. The idea was to construct some Lipschitz initial data of series form and to show that such initial data leaves $X^\beta$ immediately. Our ill-posedness result in Theorem~\ref{thm:nonexist1} and \ref{thm:nonexist2} only deals with the red set in Figure \ref{fig:wp}, but it holds for more general initial data: namely, for every smooth initial data that does not vanish on the boundary, we show that the solution must leave $C^\beta$ immediately (recall that $X^\beta$ is a subset of $C^\beta$).

\subsection{Outline of the paper}

The remainder of this paper is organized as follows. In Section \ref{sec:key}, we collect a few key estimates for the velocity. The proof of well-posedness (Theorem \ref{thm_loc1}) is given in Section \ref{sec:lwp}. Lastly, we prove nonexistence results (Theorems \ref{thm:nonexist1} and \ref{thm:nonexist2}) in Section \ref{sec:illposed}.  

\subsection*{Acknowledgments} IJ has been supported by the Samsung Science and Technology Foundation under Project Number SSTF-BA2002-04.  YY is partially supported by the NUS startup grant A-0008382-00-00 and MOE Tier 1 grant A-0008491-00-00.

\section{Estimates on the velocity}\label{sec:key}

In this section, we collect a few ``frozen-time'' estimates on the velocity. 

\subsection{Key Lemma}

The following statement shows the precise regularity of the velocity under the assumption $\tht \in X^{\alp}$. 
\begin{lemma}\label{lem:vel}
	Let $\alp \in (0,1)$ and $\tht : \overline{\bbR^2_+} \to \bbR$ with $\supp \tht \subset B(0;1)$ and $| x_1^{1-\alpha}\partial_1 \tht(x) | + | \partial_2 \tht(x)| \leq C$ for some $C>0$. Then, the velocity $u =- \nb^\perp (-\lap)^{-1+ \frac{\alp}{2}}\tht$ satisfies \begin{equation}\label{eq:u-logLip}
		\begin{split}
			\nrm{u_1}_{C^{1,1-\alpha}(\overline{\bbR^2_+})}+\nrm{\partial_2u_2}_{C^{1-\alpha}(\overline{\bbR^2_{+}})} + \nrm{\partial_1(u_2 - U_2)}_{L^{\infty}(\overline{\bbR^2_+})} \le C\left(\nrm{x_1^{1-\alpha} \partial_1\tht}_{L^\infty(\bbR^2_+)} + \nrm{\partial_2\tht}_{L^\infty(\bbR^2_+)} \right)
		\end{split}
	\end{equation} where \begin{equation}\label{eq:U2-def}
	\begin{split}
		U_2(x) := - \frac2{\alp}\int_{-\infty}^{\infty} \frac { \theta(0,y_2) }{|x-(0,y_2)|^{\alp}}\,\mathrm{d}y_2.
	\end{split}
\end{equation} Furthermore, $U_2$ satisfies \begin{equation}\label{eq:U2-est}
\begin{split}
	\nrm{\rd_2 U_2}_{L^\infty({\bbR^2_+})} \le C\nrm{\partial_2 \tht}_{L^\infty(\bbR^2_+)}
\end{split}
\end{equation} and \begin{equation}\label{eq:U2-est-unbounded}
\begin{split}
	\left| \rd_1 U_2 (x) - C_\alp x_1^{-\alp} \tht(0,x_2)  \right| \le  Cx_1^{1-\alp} \nrm{\partial_2\tht}_{L^\infty}
\end{split}
\end{equation} where \begin{equation}\label{eq:C-alp-def}
\begin{split}
	C_\alp :=  2 \int_{\bbR} \frac{\ud z}{ (1+z^2)^{\alp/2+1} } . 
\end{split}
\end{equation} In particular, in the region $\{ x_1 \ge L \}$, we have $\nrm{u}_{Lip}\le CL^{-\alp} \big(\nrm{x_1^{1-\alpha} \partial_1\tht}_{L^\infty(\bbR^2_+)} + \nrm{\partial_2\tht}_{L^\infty(\bbR^2_+)} \big)$.
\end{lemma}
\begin{remark}
	For any $x_1 \neq x'_1$, we have  \begin{equation*}
\begin{split}
	\frac {|\tht(x_1,x_2) - \tht(x'_1,x_2)|}{|x_1 - x'_1|^{\alpha}} %\leq \| x_1^{1-\alpha} \partial_1 \tht \|_{L^{\infty}} \frac {\int^{x_1}_{x'_1} \tau^{-1+\alpha}\,\ud \tau}{|x_1-x'_1|^{\alpha}} 
    \leq C\| x_1^{1-\alpha} \partial_1 \tht \|_{L^{\infty}}.
\end{split}
\end{equation*}
\end{remark}
\begin{remark}\label{rmk:SQG}
	In the case of $\alpha = 1$, we can obtain the same result with \begin{equation*}
		\begin{split}
			|u_1(x) - u_1(x')| + |(u_2 - U_2)(x) - (u_2 - U_2)(x')| \le C\| \tht \|_{C^1} |x-x'| \log \left( 10 + \frac {1}{|x-x'|} \right)
		\end{split}
	\end{equation*} 		instead of  \eqref{eq:u-logLip}.
\end{remark}
\begin{remark}
    One can replace \eqref{eq:U2-est-unbounded} by     \begin{equation*}
        \begin{split}
            \left| \rd_1 U_2 (x) - C_\alp x_1^{-\alp} \tht(x)  \right| \le  C(x_1^{1-\alp} \nrm{\partial_2\tht}_{L^\infty} + \| x_1^{1-\alpha} \partial_1 \tht \|_{L^{\infty}}).
        \end{split}
    \end{equation*} 
\end{remark}

\begin{proof} Let $\bar{\tht}$ to be the odd extension of $\tht$ in $\bbR^2$; that is, we set $\bar{\tht}(x_1,x_2)=-\tht(-x_1,x_2)$ for $x_1<0$ and $\bar{\tht}=\tht$ otherwise. Recalling the Biot--Savart law \eqref{eq:Biot-Savart-aSQG}, we have 	\begin{equation*}
		\begin{split}
			u_1(x) &= -\int_{\bbR^2} \frac{(x_2-y_2)}{|x-y|^{2+\alp}} \bar{\tht}(y) \, \ud y =  -\frac1{\alp}\int_{\bbR^2} \rd_{y_2} |x-y|^{-\alp} \bar{\tht}(y) \, \ud y \\
			& = \frac1{\alp}\int_{\bbR^2}  |x-y|^{-\alp} \rd_2\bar{\tht}(y) \, \ud y = \frac1{\alp} \int_{\bbR^2_+} \left[  \frac{1}{|x-y|^\alp} - \frac{1}{|x-\tilde{y}|^\alp}  \right] \rd_2 \tht(y) \, \ud y. 
		\end{split}
	\end{equation*} 	Then, it holds 	\begin{equation*}
		\begin{split}
			\nabla u_1(x) = -\int_{\bbR^2_+} \left[  \frac{x-y}{|x-y|^{2+\alp}} - \frac{x-\tilde{y}}{|x-\tilde{y}|^{2+\alp}}  \right] \rd_2 \tht(y) \, \ud y. 
		\end{split}
	\end{equation*} 	Given $f \in L^\infty$ with $\supp f \subset B(0;1)$, it is not difficult to see that \begin{equation*}
	\begin{split}
		\int_{ \bbR^2_+ } \frac{x-y}{|x-y|^{2+\alp}} f(y) \, \ud y, \qquad 	\int_{ \bbR^2_+ } \frac{x-\tilde{y}}{|x-\tilde{y}|^{2+\alp}} f(y) \, \ud y
	\end{split}
\end{equation*} are $C^{1-\alpha}$ functions of $x$, in the region $\overline{\bbR^2_+}$. Therefore, 	\begin{equation*}
		\begin{split}
			\nrm{u_1}_{C^{2-\alpha}(\overline{\bbR^2_+})}+\nrm{\partial_2u_2}_{C^{1-\alpha}(\overline{\bbR^2_{+}})} \le C\nrm{\partial_2\tht}_{L^\infty(\bbR^2_+)}.
		\end{split}
	\end{equation*} 	On the other hand, \begin{equation*}
		\begin{split}
			u_2(x) = -\frac1{\alp} \int_{\bbR^2_+} \left[\frac{1}{|x-y|^\alp}+\frac{1}{|x-\tilde{y}|^\alp}\right] \rd_1 \tht(y) \, \ud y  + U_2(x)
		\end{split}
	\end{equation*} with $U_2(x)$ defined as in \eqref{eq:U2-def}. Thus, it holds 	\begin{equation*}
		\begin{split}
			\partial_1 (u_2(x) - U_2(x)) = \int_{\bbR^2_+} \left[  \frac{x_1-y_1}{|x-y|^{2+\alp}} + \frac{x_1+y_1}{|\tilde{x}-y|^{2+\alp}}  \right] \rd_1 \tht(y) \, \ud y. 
		\end{split}
	\end{equation*} 	We claim $$\left| \int_{0 \leq y_1 \leq \frac {3}{2}x_1} \frac{x_1-y_1}{|x-y|^{2+\alp}} \rd_1 \tht(y) \, \ud y \right| + \left| \int_{0 \leq y_1 \leq \frac {3}{2}x_1} \frac{x_1+y_1}{|\tilde{x}-y|^{2+\alp}} \rd_1 \tht(y) \, \ud y \right| \leq C \| x_1^{1-\alpha} \partial_1 \tht \|_{L^{\infty}}$$ first. Note that $$\left| \int_{0 \leq y_1\leq \frac {1}{2}x_1} \frac {x_1-y_1}{|x-y|^{2+\alpha}} \partial_1\tht(y) \,\ud y \right|  \leq Cx_1 \| x_1^{1-\alpha}\partial_1\tht \|_{L^{\infty}} \int_{0 \leq y_1 \leq \frac {1}{2}x_1} \frac {1}{|x-y|^{2+\alpha}} \frac {1}{y_1^{1-\alpha}} \,\ud y.$$ Since change of variables $y = (x_1z_1,x_2 + x_1z_2)$ gives $$\int_{0 \leq y_1 \leq \frac {1}{2}x_1} \frac {1}{|x-y|^{2+\alpha}} \frac {1}{y_1^{1-\alpha}} \,\ud y = x_1^{-1} \int_{0 \leq z_1 \leq \frac {1}{2}} \frac {1}{|(1-z_1,z_2)|^{2+\alpha}} \frac {1}{z_1^{1-\alpha}} \,\ud z \leq Cx_1^{-1}.$$ Note that
	\begin{equation*}
	\begin{aligned}
		\left| \int_{|y_1-x_1| \leq \frac {1}{2}x_1} \frac {x_1-y_1}{|x-y|^{2+\alpha}} \partial_1\tht(y) \,\ud y \right| &\leq \int_{|y_1-x_1| \leq \frac {1}{2}x_1} \frac {|\partial_1\tht(y)|}{|x-y|^{1+\alpha}}  \,\ud y \\
		&\leq x_1^{-(1-\alpha)} \| x_1^{1-\alpha} \partial_1 \tht \|_{L^{\infty}} \int_{|y_1-x_1| \leq \frac {1}{2}x_1} \frac {1}{|x-y|^{1+\alpha}} \,\ud y.
	\end{aligned}
\end{equation*}		It is clear that $$\int_{|y_1-x_1| \leq \frac {1}{2}x_1,|y_2-x_2| \leq \frac {1}{2}x_1} \frac {1}{|x-y|^{1+\alpha}} \,\ud y \leq \int_{|y-x| \leq x_1} \frac {1}{|x-y|^{1+\alpha}} \,\ud y \leq x_1^{1-\alpha}.$$ Since $$\int_{|y_2-x_2| \geq \frac {1}{2}x_1} \frac {1}{|x-y|^{1+\alpha}} \,\ud y_2 \leq \int_{|y_2-x_2| \geq \frac {1}{2}x_1} \frac {1}{|x_2-y_2|^{1+\alpha}}\,\ud y_2 \leq C x_1^{-\alpha},$$ we have $$\int_{|y_1-x_1| \leq \frac {1}{2}x_1,|y_2-x_2| \geq \frac {1}{2}x_1} \frac {1}{|x-y|^{1+\alpha}} \,\ud y \leq Cx_1^{1-\alpha}.$$ Thus, the claim is obtained. It is left to show		\begin{equation*}
		\begin{split}
			\left| \int_{y_1 \geq \frac {3}{2}x_1} \left[  \frac{x_1-y_1}{|x-y|^{2+\alp}} + \frac{x_1+y_1}{|\tilde{x}-y|^{2+\alp}}  \right] \rd_1 \tht(y) \, \ud y \right| \leq C\| x_1^{1-\alpha} \partial_1\tht \|_{L^{\infty}}.
		\end{split}
	\end{equation*} 	For this, we write,  	\begin{equation*}
		\begin{split}
			\frac{x_1-y_1}{|x-y|^{2+\alp}} + \frac{x_1+y_1}{|\tilde{x}-y|^{2+\alp}} &= x_1 \left( \frac{1}{|x-y|^{2+\alp}} + \frac{1}{|\tilde{x}-y|^{2+\alp}} \right) + y_1 \frac {|x-y|^{2+\alpha} - |\tilde{x}-y|^{2+\alpha}}{|\tilde{x}-y|^{2+\alpha} |x-y|^{2+\alpha}}.
		\end{split}
	\end{equation*} 	It is not difficult to show that		\begin{equation*}
	\begin{aligned}
		\left| \int_{y_1 \geq \frac {3}{2}x_1} x_1 \left( \frac{1}{|x-y|^{2+\alp}} + \frac{1}{|\tilde{x}-y|^{2+\alp}} \right) \rd_1 \tht(y) \, \ud y \right| \leq C\| x_1^{1-\alpha} \partial_1\tht \|_{L^{\infty}}.
	\end{aligned}
\end{equation*}		Using that		\begin{equation}\label{eq:mvt}
	\begin{aligned}
		|x-y|^{2+\alpha} - |\tilde{x}-y|^{2+\alpha} &= \int_0^1 \frac {\ud}{\ud \tau} |\tilde{x}+(x-\tilde{x})\tau - y|^{2+\alpha} \,\ud \tau \\
		&= (2+\alpha) \int_0^1 \left( 4x_1^2\tau - 2x_1(x_1+y_1)  \right)|\tilde{x}+(x-\tilde{x})\tau - y|^{\alpha} \,\ud \tau \\
		&\leq Cx_1(x_1+y_1) |\tilde{x} - y|^{\alpha}
	\end{aligned}
\end{equation}		for any $x,y \in \bbR^2_{+}$, we have		\begin{equation*}
	\begin{aligned}
		\left| \int_{y_1 \geq \frac {3}{2}x_1} y_1 \frac {|x-y|^{2+\alpha} - |\tilde{x}-y|^{2+\alpha}}{|\tilde{x}-y|^{2+\alpha} |x-y|^{2+\alpha}} \rd_1 \tht(y) \, \ud y \right| &= \left| \int_{y_1 \geq \frac {3}{2}x_1} \frac {x_1y_1 (x_1+y_1)}{|\tilde{x}-y|^{2} |x-y|^{2+\alpha}} \rd_1 \tht(y) \, \ud y \right| \\
		&\leq Cx_1^{\alpha} \| x_1^{1-\alpha} \partial_1\tht \|_{L^{\infty}} \int_{y_1 \geq \frac {3}{2}x_1} \frac {1}{|x-y|^{2+\alpha}} \, \ud y \\
		&\leq C \| x_1^{1-\alpha} \partial_1\tht \|_{L^{\infty}}.
	\end{aligned}
\end{equation*}		This gives \eqref{eq:u-logLip}. Using that \begin{equation*}
\begin{split}
	\rd_{2} U_{2}(x) := - \frac2{\alp}\int_{-\infty}^{\infty} \frac { \rd_{2} \theta(0,y_2) }{|x-(0,y_2)|^{\alp}}\,\mathrm{d}y_2,
\end{split}
\end{equation*} we see that $\partial_2U_{2}$ is $C^{1-\alpha}$. Furthermore, by writing \begin{equation*}
\begin{split}
	\rd_1U_2(x_1,x_2) & = 2\int_{\bbR} \frac{x_1 \tht(0,y_2)}{|x-(0,y_2)|^{\alp+2}} \, \ud y_2 = 2x_1^{-\alp} \int_{\bbR} \frac{\tht(0,x_2+zx_1)}{ (1+z^2)^{\alp/2+1} } \, \ud z  \\
	& = C_\alp x_1^{-\alp} \tht(0,x_2) +   2x_1^{-\alp} \int_{\bbR} \frac{\tht(0,x_2+zx_1) - \tht(0,x_2)}{ (1+z^2)^{\alp/2+1} } \, \ud z  
\end{split}
\end{equation*} with $C_\alp$ defined in \eqref{eq:C-alp-def}, we see that the last term is bounded by \begin{equation*}
\begin{split}
	\le Cx_1^{-\alp} \nrm{\partial_2\tht}_{L^\infty} \int_{\bbR} \frac{|zx_1|}{(1+z^2)^{\alp/2+1}} \, \ud z \le Cx_1^{1-\alp} \nrm{\partial_2\tht}_{L^\infty}. 
\end{split}
\end{equation*} This finishes the proof. 
\end{proof}

We can also prove the following lemma.
\begin{lemma}\label{lem:vel2}
	Let $\alp \in (0,1)$ and  $\tht \in C^{\alpha}_c(\overline{\bbR^2_+})$ with $\supp \tht \subset B(0;1)$. Assume that $\varphi \in C^{\infty}_c(\overline{\bbR^2_{+}})$ be a smooth bump function with $\varphi(x) = 1$ for $x \in B(0;16)$. Then, the velocity $u =- \nb^\perp (-\lap)^{-1+ \frac{\alp}{2}}\tht$ satisfies \begin{equation}\label{eq:u-logLip2}
		\begin{gathered}
			|u_1(x) - u_1(x')| + |u_2(x) - u_2(x') - \tht(x)(f(x) - f(x'))| \le C|x-x'|\log\left(10+\frac {1}{|x-x'|}\right) \| \tht \|_{C^{\alpha}(\bbR^2_{+})},
		\end{gathered}
	\end{equation}		
where \begin{equation}\label{eq:U2-def2}
	\begin{split}
		f(x) := -\frac {2}{\alpha}\int_{\bbR} \frac {1}{|x-(0,z)|^{\alpha}} \varphi(0,z) \,\ud z.
	\end{split}
\end{equation} Furthermore, it holds \begin{equation}\label{eq:U2-est2}
\begin{split}
	\left|\partial_2f(x)\right| \le C
\end{split}
\end{equation} and \begin{equation}\label{eq:U2-est-unbounded2}
\begin{split}
	\partial_1f(x) = C_\alp(x) x_1^{-\alp},
\end{split}
\end{equation} where \begin{equation}\label{eq:C-alp-def2}
\begin{split}
	C_\alp(x) :=  2 \int_{\bbR} \frac{1}{ (1+z^2)^{1+\alp/2} }  \varphi(0,x_2+x_1z) \,\ud z. 
\end{split}
\end{equation}
\end{lemma}

\begin{comment}
\begin{remark}
    There exists a constant $C > 1$ such that $\frac {1}{C} \leq C_{\alpha}(x) \leq C$ for all $x \in B(0;1)$. In particular, in the region $\{ x_1 \ge L \}$, $u$ is log-Lipshcitz with the constant $CL^{-\alp} \nrm{\tht}_{C^{\alpha}}$. 
\end{remark}
\end{comment}

%\begin{remark}\label{rmk:beta}
%    One can have Lipschitz bound of \eqref{eq:u-logLip2} if $\| \tht \|_{C^{\beta}}$ is assumed for $\beta > \alpha$.
%\end{remark}
\begin{remark}\label{rmk:reg_euler}
    One can have the same result for $\alpha=0$, if we replace \eqref{eq:U2-def2} and \eqref{eq:C-alp-def2} by        \begin{equation*}
        \begin{split}
            f(x) := -2\int_{\bbR} \log(|x-(0,z)|) \varphi(0,z) \,\ud z, \qquad             C_\alp(x) :=  2 \int_{\bbR} \frac{1}{ 1+z^2}  \varphi(0,x_2+x_1z) \,\ud z,
        \end{split}
    \end{equation*}   respectively.
\end{remark}

%\begin{lemma}
%	For $\tht\in C^{\alp+\varepsilon}$, $u_1$ is Lipschitz and $u_2$ is Lipschitz part plus $\tht(x)(\bar{f}(x_2)+ f(x))$, where $\bar{f}(x_2)+ f(x)$ is smooth in $x_2$ and $f$ satisfies $|f(x)|\lesssim x_1^{1-\alp}$, $|\rd_{1}f(x)| \lesssim x_1^{-\alp}$. 
%\end{lemma}

\subsection{Velocity $L^2$ estimates}

We prove two $L^2$-type estimates for the velocity. 
\begin{lemma}\label{lem:L2}
	Let $g \in L^2(\bbR^2_{+})$ be compactly supported, and let $v = -\nb^\perp (-\lap)^{-1+\alp/2}g$. Then, we have \begin{equation}\label{eq:u2}
		\begin{split}
			\nrm{v_2}_{L^2} \le C\nrm{g}_{L^2} 
		\end{split}
	\end{equation} and \begin{equation}\label{eq:u1}
		\begin{split}
			\nrm{ x_{1}^{\alp-1} v_1}_{L^2} \le C\nrm{g}_{L^2} 
		\end{split}
	\end{equation} for some $C>0$ depending only on the diameter of the support of $g$. 
\end{lemma}

\begin{proof}
	We first prove \eqref{eq:u2}: using the Biot--Savart law, we bound	\begin{equation*}
		\begin{aligned}
			\| v_2 \|_{L^2} &\leq \left\| \int_{\bbR^2} \frac{(x_1-y_1)}{|x-y|^{2+\alp}} g(y) \, \ud y \right\|_{L^2} \leq \| g \|_{L^2} \int_{B(0;R)} \frac {1}{|x|^{1+\alpha}} \,\ud x \leq C\| g \|_{L^2}.
		\end{aligned}
	\end{equation*} 
	
	Next, to prove \eqref{eq:u1}, we begin with		\begin{equation*}
		\begin{aligned}
			\| x_1^{-(1-\alpha)}v_{1} \|_{L^2} &= \left\| x_1^{-(1-\alpha)} \int_{\bbR^2_{+}} \left( \frac{x_2-y_2}{|x-y|^{2+\alp}} - \frac{x_2-y_2}{|\tilde{x}-y|^{2+\alp}} \right) g(y) \, \ud y \right\|_{L^2}.
		\end{aligned}
	\end{equation*}		Recalling \eqref{eq:mvt}, we write		\begin{equation*}
		\begin{aligned}
			\left| \frac{x_2-y_2}{|x-y|^{2+\alp}} - \frac{x_2-y_2}{|\tilde{x}-y|^{2+\alp}} \right| \leq\frac { Cx_1(x_1+y_1)(x_2-y_2)|\tilde{x}-y|^{\alpha}}{|x-y|^{2+\alp}|\tilde{x}-y|^{2+\alp}} \leq \frac {Cx_1}{|x-y|^{1+\alp}|\tilde{x}-y|} .
		\end{aligned}
	\end{equation*}		Thus, it follows		\begin{equation*}
		\begin{aligned}
			\| x_1^{-(1-\alpha)}v_{1} \|_{L^2} &\leq \left\| x_1^{\alpha} \int_{\bbR^2_{+}} \frac {|g(y)|}{|x-y|^{1+\alp}|\tilde{x}-y|} \, \ud y \right\|_{L^2}.
		\end{aligned}
	\end{equation*}	Using the Minkowski integral inequality and applying the change of variables $y = (x_1z_1,x_2 + x_1z_2)$, we obtain	\begin{equation*}
		\begin{aligned}
			\left\| x_1^{\alpha} \int_{\bbR^2_{+}} \frac {|g(y)|}{|x-y|^{1+\alp}|\tilde{x}-y|} \, \ud y \right\|_{L^2} &= \left\| \int_{\bbR^2_{+}} \frac {|g(x_1z_1,x_2 + x_1z_2)|}{|(1-z_1,z_2)|^{1+\alpha}|(1+z_1,z_2)|} \,\ud z \right\|_{L^2} \\
			&\leq \| g\|_{L^2} \int_{\bbR^2} \frac {1}{\sqrt{z_1} |(1-z_1,z_2)|^{1+\alpha}|(1+z_1,z_2)|} \,\ud z \\
			&\leq C \| g\|_{L^2}.
		\end{aligned}
	\end{equation*} This finishes the proof. 
\end{proof}

\section{Local well-posedness}\label{sec:lwp}

In this section, we shall prove Theorem \ref{thm_loc1}. We fix some $0<\alp\le 1/2$ and $\alp\le\bt\le1-\alp$. We proceed in several steps. 

\medskip

\noindent \textbf{1. A priori estimates}. We take $\tht_0 \in X^{\bt}$ and assume for simplicity that $\supp \tht_0 \subset B(0;1)$. To begin with, we have $\nrm{u}_{L^\infty} $ finite so that the support of $\tht$ will be contained in a ball of radius $\lesssim 1$. In particular, on the support of $\tht$, we have $x_1 \lesssim 1$. Then, we consider the derivative estimates for the (hypothetical) solution to \eqref{eq:SQG}.

We first estimate the derivative in $x_2$: 
\begin{equation}\label{eq:rd2}
	\begin{split}
		\rd_t (\rd_2\tht) + u\cdot\nb (\rd_2\tht) = - \rd_2 u \cdot \nb \tht. 
	\end{split}
\end{equation} We note \begin{equation*}
	\begin{split}
		\frac{\ud}{\ud t} \nrm{\rd_2\tht}_{L^\infty} \le \nrm{\rd_2 u_1\rd_1\tht}_{L^\infty} + \nrm{\rd_2u_2\rd_2\tht}_{L^\infty} .
	\end{split}
\end{equation*} Observe that $\rd_2u_1 \in C^{1-\alp} $ using that $\rd_2\tht$ is bounded. Since $\rd_2u_1$ vanishes on the boundary, we have that $|x_1^{\alp-1}\rd_2u_1|$ is bounded. This allows \begin{equation*}
	\begin{split}
		\nrm{\rd_2 u_1\rd_1\tht}_{L^\infty} \lesssim \nrm{x_1^{\alp-1}\rd_2 u_1}_{L^\infty}  \nrm{ x_1^{1-\alp} \rd_1\tht}_{L^\infty} \lesssim \nrm{ \rd_2\tht}_{L^\infty}   \nrm{ x_1^{1-\alp} \rd_1\tht}_{L^\infty}.
	\end{split}
\end{equation*}  On the other hand, $\nrm{\rd_2u_2}_{C^{1-\alp}} \lesssim \nrm{\partial_2\tht}_{L^{\infty}}$. Hence \begin{equation}\label{eq:2tht}
	\begin{split}
		\frac{\ud}{\ud t} \nrm{\rd_2\tht}_{L^\infty} \lesssim ( \nrm{ x_1^{1-\alpha} \rd_1\tht}_{L^\infty}+ \nrm{\partial_2\tht}_{L^{\infty}}) \nrm{ \rd_2\tht}_{L^\infty} . 
	\end{split}
\end{equation}

Now we consider 
\begin{equation}\label{eq:rd1}
	\begin{split}
		\rd_t (\rd_1\tht) + u\cdot\nb (\rd_1\tht) = - \rd_1 u \cdot \nb \tht. 
	\end{split}
\end{equation}
Multiplying both terms by $x_1^{1-\beta}$ gives		\begin{equation*}
	\begin{split}
		\rd_t (x_1^{1-\beta}\rd_1\tht) + u\cdot\nb (x_1^{1-\beta}\rd_1\tht) = - x_1^{1-\beta}\rd_1 u \cdot \nb \tht + (1-\beta) u_1 x_1^{-\beta}\rd_1\tht . 
	\end{split}
\end{equation*} The last term is easy to control, using that $u_1$ vanishes on the boundary: $$\nrm{  u_1 x_1^{-\beta}\rd_1\tht}_{L^\infty} \le \nrm{\rd_1u_1}_{L^\infty} \nrm{ x_1^{1-\beta}\rd_1\tht}_{L^{\infty}}.$$ Note \begin{equation*}
	\begin{split}
		\nrm{ x_1^{1-\beta}\rd_1 u_1\rd_1\tht}_{L^\infty} \le \nrm{\rd_1u_1}_{L^\infty} \nrm{ x_1^{1-\beta}\rd_1\tht}_{L^\infty}
	\end{split}
\end{equation*}
and
\begin{equation*}
	\begin{split}
		\nrm{ x_1^{1-\beta}\rd_1 u_2\rd_2\tht}_{L^\infty} &\le \nrm{ x_1^{1-\beta}\rd_1 u_2}_{L^\infty} \nrm{\rd_2\tht}_{L^\infty} \\
		&\lesssim (\| x_1^{1-\alpha} \partial_1 \tht \|_{L^{\infty}} + \| \partial_2 \tht \|_{L^{\infty}})\| \partial_2 \tht \|_{L^{\infty}} + \| x_1^{1-\beta-\alpha} \tht \|_{L^{\infty}} \nrm{\rd_2\tht}_{L^\infty}.
	\end{split}
\end{equation*} We used \eqref{eq:u-logLip} and \eqref{eq:U2-est-unbounded} in the last inequality. It is necessary that $\alp \leq \frac {1}{2}$ to hold $\beta \geq\alp$ and $1-\alpha-\beta \geq 0$ simultaneously. 

Combining the inequalities, we obtain \begin{equation}\label{eq:energy}
	\begin{split}
		\frac{\ud}{\ud t} ( \nrm{x_1^{1-\bt} \rd_1\tht}_{L^\infty} + \nrm{\rd_2\tht}_{L^\infty})  \lesssim  ( \nrm{x_1^{1-\alpha} \rd_1\tht}_{L^\infty} + \nrm{\rd_2\tht}_{L^\infty})  ( \nrm{\tht}_{L^\infty} + \nrm{x_1^{1-\bt} \rd_1\tht}_{L^\infty} + \nrm{\rd_2\tht}_{L^\infty}). 
	\end{split}
\end{equation}

Lastly, we note that $\nrm{\tht}_{L^\infty}=\nrm{\tht_0}_{L^\infty}$ and $\nrm{\tht}_{C^\bt} \lesssim \nrm{\tht}_{L^\infty} + \nrm{x_1^{1-\bt} \rd_1\tht}_{L^\infty} + \nrm{\rd_2\tht}_{L^\infty}$. This gives that there exists $T = T(\nrm{\tht_0}_{X^{\alpha}},|\supp \tht_0|)>0$ such that on $[0,T]$, $\nrm{\tht(t,\cdot)}_{X^{\bt}} \lesssim \nrm{\tht_0}_{X^{\bt}}.$

\medskip 

Now we collect some estimates on the flow map: On the time interval $[0,T]$, we can find a unique solution to the ODE for any $x\in \bbR^2_+$: \begin{equation*}
	\frac{\ud}{\ud t} \Phi(t,x)= u(t,\Phi(t,x)), \qquad \Phi(0,x)=x. 
\end{equation*} While this is not trivial as $u$ is not uniformly Lipschitz on the half plane, the point is that $u_1$ is uniformly Lipschitz, which gives together with  $u_1(t,(0,x_2))=0$ that the first component of the flow map satisfies the estimate \begin{equation*}
	\begin{split}
		\left| \frac{\ud}{\ud t} \Phi_1(t,x) \right| \le C\| \tht \|_{X^{\alpha}} \Phi_1(t,x)
	\end{split}
\end{equation*} by the mean value theorem. This gives in particular that \begin{equation}\label{eq:Phi-1-est}
	\begin{split}
		x_1 \exp(-C\| \tht_0 \|_{X^\alpha}t) \le	\Phi_1(t,x) \le x_1 \exp(C\| \tht_0 \|_{X^\alpha}t).
	\end{split}
\end{equation} Since $u$ is uniformly Lipschitz away from the boundary $\{ x_1=0\} $, this shows that the flow map $\Phi$ is well-defined on $[0,T]$. Moreover, Lemma~\ref{lem:vel} gives for any fixed $x \in \bbR^2_{+}$ that $\sup_{t \in [0,T]}|\Phi(t,x) - \Phi(t,x')| \to 0$ as $x' \to x$.

It is not difficult to show that $\Phi$ is differentiable in $x$ almost everywhere, with the following a priori bound for $\nabla\Phi(t,x)$:     \begin{equation}\label{eq:flow-diff}
	\begin{split}
		|\nabla \Phi(t,x)| \leq e^{C(x_1^{-\alpha}\tht_0(x) + \| \tht_0 \|_{X^{\alpha}}) t}.
	\end{split}
\end{equation} For the proof, fix some $x=\bbR^2_{+}$ and let $x'=(x'_1,x_2)$ with $x'_1> \frac{x_1}{2}$. Then, we have        \begin{equation*}
	\begin{split}
		\frac{\ud}{\ud t} \frac{\Phi(t,x) - \Phi(t,x')}{x_1-x'_1} = \frac{u(t,\Phi(t,x)) - u(t,\Phi(t,x'))}{x_1-x'_1}.
	\end{split}
\end{equation*}     Note that       \begin{equation*}
	\begin{gathered}
		\frac{u(t,\Phi(t,x)) - u(t,\Phi(t,x'))}{x_1-x'_1} = \frac{u(t,\Phi(t,x)) - u(t,\Phi_1(t,x'),\Phi_2(t,x))}{\Phi_1(t,x)-\Phi_1(t,x')} \frac{\Phi_1(t,x)-\Phi_1(t,x')}{x_1-x'_1} \\
		+ \frac{u(t,\Phi_1(t,x'),\Phi_2(t,x)) - u(t,\Phi(t,x'))}{\Phi_2(t,x)-\Phi_2(t,x')} \frac{\Phi_2(t,x)-\Phi_2(t,x')}{x_1-x'_1}.
	\end{gathered}
\end{equation*}     On the other hand, we consider the following linear ODE system:      \begin{equation}\label{eq:ODE}
	\begin{gathered}
		\frac{\ud}{\ud t} \partial_1 \Phi(t,x) = \begin{pmatrix}
			\partial_1u_1(t,\Phi(t,x)) & \partial_2 u_1(t,\Phi(t,x)) \\
			\partial_1u_2(t,\Phi(t,x)) & \partial_2 u_2(t,\Phi(t,x)) \end{pmatrix} \partial_1\Phi(t,x).
	\end{gathered}
\end{equation}     It is obvious that there exists a unique solution $\partial_1\Phi(t,x) \in \operatorname{Lip}(0,T)$ with the initial data $\partial_1\Phi(0,x) = e_1$. Since \eqref{eq:Phi-1-est} implies $\inf_{t \in [0,T]} \Phi_1(t,x) \geq \frac{x_1}{C}$ for some $C>0$, not depending on the choice of $x$, we can deduce  with Lemma~\ref{lem:vel}     \begin{equation}\label{eq:nb_u}
	\begin{split}
		\sup_{t \in [0,T]} | \nabla u(t,\Phi(t,x)) | \leq C(x_1^{-\alpha}\tht_0(x) + \| \tht_0 \|_{X^{\alpha}}).
	\end{split}
\end{equation}      Thus, it follows       \begin{equation*}
	\begin{split}
		|\partial_1\Phi(t,x)| \leq e^{C(x_1^{-\alpha}\tht_0(x) + \| \tht_0 \|_{X^{\alpha}}) t}.
	\end{split}
\end{equation*}     Moreover, using     \begin{multline*}
	\lim_{x'_1 \to x_1} \Bigg(\left| \partial_1 u(t,\Phi(t,x)) - \frac{u(t,\Phi(t,x)) - u(t,\Phi_1(t,x'),\Phi_2(t,x))}{\Phi_1(t,x)-\Phi_1(t,x')} \right| \\
	+ \left| \partial_2 u(t,\Phi(t,x)) - \frac{u(t,\Phi_1(t,x'),\Phi_2(t,x)) - u(t,\Phi(t,x'))}{\Phi_2(t,x)-\Phi_2(t,x')} \right| \Bigg) = 0
\end{multline*}     and the uniform bound \eqref{eq:nb_u}, we can have by Lebesgue's dominated convergence theorem that $$\lim_{x'_1 \to x_1} \frac{\Phi(t,x)-\Phi(t,x')}{x_1-x'_1} = \partial_1 \Phi(t,x)$$ for a.e. $x$ and $t$. Similarly, one can repeat the above process for $x'=(x_1,x'_2)$ to obtain           \begin{equation*}
	\begin{split}
		\lim_{x'_2 \to x_2} \frac{\Phi(t,x)-\Phi(t,x')}{x_2-x'_2} = \partial_2 \Phi(t,x), \qquad |\partial_2\Phi(t,x)| \leq e^{C(x_1^{-\alpha}\tht_0(x) + \| \tht_0 \|_{X^{\alpha}}) t}.
	\end{split}
\end{equation*} %For those who require more detail, we also offer the following lemma:
The details are provided by the following lemma: 
\begin{lemma}
	Let $A(t)$ and $B(t)$ be $2 \times 2$ matrices such that        \begin{equation*}
		A(t) := \begin{pmatrix}
			a_{11}(t) &a_{12}(t) \\
			a_{21}(t) &-a_{11}(t)
		\end{pmatrix}, \qquad B^{\varepsilon}(t) := A(t) +        \begin{pmatrix}
			b_{11}(t) &b_{12}(t) \\
			b_{21}(t) &b_{22}(t)
		\end{pmatrix}, \qquad \int_0^T |b_{ij}(t)| \,\ud t \leq \varepsilon
	\end{equation*}     for some $T>0$ and $\varepsilon > 0$. Let $y(t)$ and $y^\varepsilon(t)$ be solutions to the linear ODE systems     \begin{equation*}
		\frac{\ud}{\ud t}y(t) = A(t)y(t), \qquad \frac{\ud}{\ud t}y^\varepsilon(t) = B^{\varepsilon}(t)y^\varepsilon(t), \qquad y(0) = y_0, \qquad y^\varepsilon(0) = y^\varepsilon_0,
	\end{equation*}     where $|y_0 - y^\varepsilon_0| \leq \varepsilon$. Then, there exists a constant $C>0$ not depending on $\varepsilon$ such that       \begin{equation*}
		\sup_{t \in [0,T]}|y(t) - y^\varepsilon(t)| \leq C\varepsilon \left(1+\sup_{t \in [0,T]} |y(t)| \right) \exp \left(C\varepsilon + \int_0^T |A(t)| \,\ud t\right).
	\end{equation*}
\end{lemma}
\begin{proof}
	We have     \begin{equation*}
		\frac{\ud}{\ud t}(y(t) - y^\varepsilon(t)) = B^{\varepsilon}(t)(y(t) - y^\varepsilon(t)) - (B(t) - A(t)) y(t).
	\end{equation*}     Integrating the both sides gives     \begin{equation*}
		y(t) - y^\varepsilon(t) = y_0 - y^\varepsilon_0 + \int_0^t B^{\varepsilon}(\tau)(y(\tau) - y^\varepsilon(\tau)) \,\ud \tau - \int_0^t (B^{\varepsilon}(\tau) - A(\tau)) y(\tau) \,\ud \tau.
	\end{equation*}     Thus, we have     \begin{equation*}
		|y(t) - y^\varepsilon(t)| \leq \varepsilon  + C\varepsilon \sup_{\tau \in [0,t]} |y(\tau)| + \int_0^t \left(|A(\tau)|+|(B^{\varepsilon}-A)(\tau)|\right)|y(\tau) - y^\varepsilon(\tau)| \,\ud \tau.
	\end{equation*}     Using Gr\"onwall's inequality, we can complete the proof.
\end{proof}

From the above estimates, we have that $\tht(t,x)$ is differentiable in $x$ almost everywhere. Moreover, for any fixed $x \in \bbR^2_{+}$, we can show that $$\left| \tht(t,\Phi(t,x)) - \tht(t',\Phi(t,x))\right| = \left| \tht(t',\Phi(t',x)) - \tht(t',\Phi(t,x)) \right| \leq C \sup_{t \in [0,T]} \| \tht(t,\cdot) \|_{X^{\beta}}^2 |t-t'|$$ for all $t,t' \in (0,T)$. Therefore, the solution is differentiable in $t$ almost everywhere, and $$\theta \in \mathrm{Lip}(0,T;L^{\infty}) \cap L^{\infty}(0,T;X^{\beta}).$$ 
Also note that $$\partial_j (\theta(t,\Phi(t,x))) = \nabla \tht(t,\Phi(t,x)) \partial_j \Phi(t,x) = \partial_j\theta_0(x)$$ for all $x \in \bbR^2_{+}$. The continuity of $\partial_j \tht_0$ is assumed in $\bbR^2_{+}$. We recall \eqref{eq:ODE} with the continuity of $\nabla u_1$ and $\partial_2 u_2$. Since $$\partial_1 u_2(x) = \partial_1 U_2 + \int_{\bbR^2_+} \left[  \frac{x_1-y_1}{|x-y|^{2+\alp}} + \frac{x_1+y_1}{|\tilde{x}-y|^{2+\alp}}  \right] \rd_1 \tht(y) \, \ud y,$$ where the terms on the right-hand side are continuous in $x$, we obtain that $\nabla \Phi(t,x)$ is continuous in $\bbR^2_{+}$. Hence, $\nabla \tht(t,x)$ is continuous in $x$ except the case of $\nabla \Phi(t,x) = 0$. This finishes the proof of a priori estimates.

\medskip

\noindent \textbf{2. Uniqueness}. Now, we show that the solution in $L^\infty_t X^{\bt}$ is unique. 
Suppose that $\tht$ and $\tht'$ are solutions to \eqref{eq:SQG} with the same initial data $\tht_0$ defined on the time interval $[0,\delta]$ for some $\dlt>0$. From $$\partial_t(\tht -\tht') + u'\cdot \nabla (\tht - \tht') = -(u - u') \cdot \nabla \tht,$$ we have $$\frac {\ud}{\ud t} \int_{\bbR^2_{+}} |\tht - \tht'|^2 \,\ud x = -\int_{\bbR^2_{+}} (u_1 - u'_1) \partial_1 \tht (\tht - \tht') \,\ud x -\int_{\bbR^2_{+}} (u_2 - u'_2) \partial_2 \tht (\tht - \tht') \,\ud x.$$ Note that		\begin{equation*}
	\begin{aligned}
		\left| -\int_{\bbR^2_{+}} (u_2 - u'_2) \partial_2 \tht (\tht - \tht') \,\ud x \right| &\leq \| \partial_2 \tht \|_{L^{\infty}} \| u_2 - u'_2 \|_{L^2} \| \tht - \tht' \|_{L^2} \leq C\| \partial_2 \tht \|_{L^{\infty}} \| \tht - \tht' \|_{L^2}^2,
	\end{aligned}
\end{equation*}	where we have applied \eqref{eq:u2} with $g=\tht-\tht'$. On the other hand, we estimate the other term as follows:	\begin{equation*}
	\begin{aligned}
		\left| -\int_{\bbR^2_{+}} (u_1 - u'_1) \partial_1 \tht (\tht - \tht') \,\ud x \right| &\leq \| x_1^{1-\alpha} \partial_1 \tht \|_{L^{\infty}} \| x_1^{-(1-\alpha)} (u_1 - u'_1) \|_{L^2} \| \tht - \tht' \|_{L^2} \\
		&\leq C \| x_1^{1-\alpha} \partial_1 \tht \|_{L^{\infty}} \| \tht - \tht' \|_{L^2}^{2}, 
	\end{aligned}
\end{equation*}	 where we have used \eqref{eq:u1} this time. 
Combining the above estimates gives that $$\frac {\ud}{\ud t} \int_{\bbR^2_{+}} |\tht - \tht'|^2 \,\ud x \leq C(\| x_1^{1-\alpha} \partial_1\tht \|_{L^{\infty}} + \| \partial_2 \tht \|_{L^{\infty}}) \int_{\bbR^2_{+}} |\tht - \tht'|^2 \,\ud x.$$ Since $\bt\ge\alp$, $\tht = \tht'$ must hold on the time interval $[0,\delta]$ by Gr\"onwall's inequality. This completes the proof of uniqueness. 

\medskip

\noindent \textbf{3. Existence}. The existence of the solution can be proved by an iteration argument. We are going to define the sequence of functions $\tht^{(n)}$ which is uniformly bounded in $L^\infty([0,T];X^{\bt})$ where $T>0$ is determined only by $\nrm{\tht_0}_{X^{\bt}}$. To this end, we first set $\tht^{(0)} \equiv \tht_0$ for all $t$ and consider \begin{equation}\label{eq:iteration}
	\left\{
	\begin{aligned}
		\rd_t \tht^{(n+1)} + u^{(n)} \cdot \nb \tht^{(n+1)} = 0,& \\
		u^{(n)} = \nb^\perp(-\lap)^{-1+\alp/2}\tht^{(n)}, & \\
		\tht^{(n+1)}(t=0)=\tht_{0}. & 
	\end{aligned}
	\right.
\end{equation}
We can prove the following claims inductively in $n$: there exists some $T>0$ such that for $t \in [0,T]$,\begin{itemize}
	\item the flow map $\Phi^{(n)}$ corresponding to $u^{(n)}$ is well-defined as a homeomorphism of $\overline{\bbR^2_+}$;
	\item $\tht^{(n+1)}(t,\Phi^{(n)}(t,x)) = \tht_{0}(x)$ and $\tht^{(n+1)}(t,x) = \tht_{0} \circ (\Phi^{(n)}_{t})^{-1}(x)$;
	\item $\sup_{t \in [0,T]} \nrm{\tht^{(n+1)}(t,\cdot)}_{X^{\bt}} \le 10 \nrm{\tht_{0}}_{X^{\bt}}$. 
\end{itemize}
The above statements can be proved as follows: assume that we have $\sup_{t \in [0,T]} \nrm{\tht^{(n)}(t,\cdot)}_{X^{\bt}} \le 10 \nrm{\tht_{0}}_{X^{\bt}}$ for some $n\ge0$. Then, $u^{(n)}$ satisfies the bounds in Lemma \ref{lem:vel}. Based on these bounds, one can solve uniquely the following ODE for any $x \in \bbR^2_+$ uniformly in the time interval $[0,T]$: \begin{equation*}
	\frac{d}{dt} \Phi^{(n)}(t,x) = u^{(n)}(t,\Phi^{(n)}(t,x)) , \qquad \Phi^{(n)}(t,x)=x. 
\end{equation*} Furthermore, $\Phi^{(n)}(t,x)$ is differentiable a.e., with the well-defined inverse $(\Phi^{(n)}_{t})^{-1}(x)$ which is again differentiable a.e. Therefore, we can define $\tht^{(n+1)}(t,x) := \tht_{0} \circ (\Phi^{(n)}_{t})^{-1}(x)$, which is easily shown to be a solution to \eqref{eq:iteration}. Then, applying the a priori estimates to $\tht^{(n+1)}$, we can derive \begin{equation*}
	\frac{d}{dt} \nrm{ \tht^{(n+1)} }_{X^{\bt}} \lesssim \nrm{ \tht^{(n)} }_{X^{\bt}}\nrm{ \tht^{(n+1)} }_{X^{\bt}} \lesssim \nrm{\tht_{0}}_{X^{\bt}} \nrm{ \tht^{(n+1)} }_{X^{\bt}}, 
\end{equation*} where the implicit constants are independent of $n$. Therefore, we obtain $\sup_{t \in [0,T]} \nrm{\tht^{(n+1)}(t,\cdot)}_{X^{\bt}} \le 10 \nrm{\tht_{0}}_{X^{\bt}}$ by possibly shrinking $T$ if necessary (but in a way which is independent of $n$). 

Now we can prove that the sequence $\left\{ \tht^{(n)} \right\}_{n\ge0} $ is Cauchy in $L^2$. For this we write $D^{(n)} = \tht^{(n+1)}-\tht^{(n)}$ and $v^{(n)} = u^{(n)}-u^{(n-1)}$ for simplicity where $n\ge1$. We have \begin{equation*}
	\frac{d}{dt} D^{(n)} + u^{(n)}\cdot \nb D^{(n)} = - v^{(n)} \cdot \nb \tht^{(n+1)}. 
\end{equation*} Integrating against $D^{(n)}$ in space and applying Lemma \ref{lem:L2} for $v^{(n)}$ and $D^{(n-1)}$, we obtain that \begin{equation*}
	\frac{d}{dt} \nrm{ D^{(n)} }_{L^2}^{2} \lesssim \nrm{\tht_{0}}_{X^{\bt}} \nrm{ D^{(n-1)} }_{L^2} \nrm{ D^{(n)} }_{L^2}. 
\end{equation*} Since $D^{(n)}(t=0)=0$, we obtain from the above that  \begin{equation*}
	\nrm{ D^{(n)} (t,\cdot) }_{L^2} \lesssim t\nrm{\tht_{0}}_{X^{\bt}} \sup_{t\in [0,T]}\nrm{ D^{(n-1)} (t,\cdot)}_{L^2} 
\end{equation*} so that by shrinking $T$ if necessary to satisfy $T\nrm{\tht_{0}}_{X^{\bt}} \ll 1$, we can inductively prove \begin{equation*}
	\sup_{t\in [0,T]}\nrm{ D^{(n)} (t,\cdot)}_{L^2} \lesssim 2^{-n}
\end{equation*} for all $n\ge1$. 

From the above, we have that for each $t \in [0,T]$, $\tht^{(n)}$ converges in $L^2$ to a function which we denote by $\tht$. Since $C^{\bt}$ is precompact in $C^{0}$, we obtain from the uniform $X^{\bt}$ bound of $\tht^{(n)}$  that actually $\tht \in L^\infty([0,T];C^\bt)$ and $\tht^{(n)}\to\tht$ in $C^\gmm$ for any $0\le \gmm <\bt$. This uniform convergence shows that $\tht \in X^{\bt}$ as well. From this, we obtain uniform convergence $u^{(n)} \to u$ and $\Phi^{(n)} \to \Phi$ for some $u$ and $\Phi$, and $\Phi$ can be shown to be the flow map corresponding to $u$. Lastly, taking the limit $n\to\infty$ in the relation $\tht^{(n+1)}(t,x) = \tht_{0} \circ (\Phi^{(n)}_{t})^{-1}(x)$ gives $\tht(t,x)=\tht_{0} \circ \Phi_t^{-1}(x)$. This shows that $\tht \in L^\infty([0,T];X^{\bt})$ is a solution to \eqref{eq:SQG} with initial data $\tht_{0}$. This finishes the proof of existence. Using $\tht(t,x)=\tht_{0} \circ \Phi_t^{-1}(x)$ and the regularity of the flow, it is not difficult to show at this point that $\tht \in C([0,T];C^{\bt'})$ for any $\bt'<\bt$. 

\medskip

\noindent \textbf{4. Blow-up criteria}. Firstly, we remark that the existence of time $T>0$ is depending on $\| \tht_0 \|_{X^{\alpha}}$, not $\| \tht_0 \|_{X^{\beta}}$, see \eqref{eq:energy}. Recalling the estimate 
\begin{equation}\label{eq:x1tht}
    \frac{\ud}{\ud t} \| x_1^{1-\alpha} \partial_1 \theta \|_{L^{\infty}} \lesssim \| x_1^{1-\alpha} \partial_1 \theta \|_{L^{\infty}} \| \partial_2 \tht \|_{L^{\infty}} + \| \tht \|_{L^{\infty}} \| \partial_2 \tht \|_{L^{\infty}} + \| \partial_2 \tht \|_{L^{\infty}}^2,
\end{equation} we can conclude that if $T<\infty$ and $$\sup_{t \in [0,T]} \| \partial_2 \tht(t) \|_{L^{\infty}} < \infty,$$ then $T$ is not the maximal time of existence. Now, we assume that the maximal time $T$ is finite. Then, by the above argument, we have $\lim_{t \to T} \| \partial_2\tht(t) \|_{L^{\infty}} = \infty$. Suppose that there exist $\eta \in (0,1)$ and $C>0$ such that 
\begin{equation}\label{ass}
    \sup_{t \in [0,T)} (T-t)^{\eta} \| \partial_2 \tht(t) \|_{L^{\infty}} \leq C.
\end{equation} Then, from \eqref{eq:x1tht}, we have $$\frac{\ud}{\ud t} \| x_1^{1-\alpha} \partial_1 \theta \|_{L^{\infty}} \lesssim \| x_1^{1-\alpha} \partial_1 \theta \|_{L^{\infty}} (T-t)^{-\eta} + \| \tht \|_{L^{\infty}} (T-t)^{-\eta} + (T-t)^{-2\eta}.$$ 
By Gr\"onwall's inequality, we have $$\| x_1^{1-\alpha} \partial_1 \theta(t) \|_{L^{\infty}} \lesssim \| x_1^{1-\alpha} \partial_1 \theta_0 \|_{L^{\infty}} + (T-t)^{-2\eta+1} + \| \tht_0 \|_{L^{\infty}}.$$ Integrating over time, we obtain $$\int_0^T \| x_1^{1-\alpha} \partial_1\tht(t) \|_{L^{\infty}} < \infty.$$ Next, we recall \eqref{eq:2tht} $$\frac{\ud}{\ud t} \nrm{\rd_2\tht}_{L^\infty} \lesssim ( \nrm{ x_1^{1-\alpha} \rd_1\tht}_{L^\infty}+ \nrm{\partial_2\tht}_{L^{\infty}}) \nrm{ \rd_2\tht}_{L^\infty} .$$ Since $\| \partial_2 \tht(t) \|_{L^{\infty}} \in L^1(0,T)$ holds by \eqref{ass}, Gr\"onwall's inequality gives that $\sup_{t \in [0,T]} \| \partial_2 \tht(t) \|_{L^{\infty}} < \infty$, which contradics the assumption \eqref{ass}.

\section{Proof of illposedness}\label{sec:illposed}

\begin{proof}[Proof of Theorem \ref{thm:nonexist1}]
	We take the initial data $\tht_0$ satisfying the assumptions of Theorem \ref{thm:nonexist1}. Then by Theorem~\ref{thm_loc1}, we obtain a unique solution $\tht \in L^{\infty}([0,\delta];X^{1-\alpha})$ for some $\delta>0$. Note that $\delta$ can be replaced by a smaller one if we need. From the assumption that $\theta_0$ does not vanish on the boundary, we have a boundary point $x_0 = (0,a)$ such that $\tht_0(x_0) \ne 0$, $\rd_2\tht_0(x_0) \ne 0$. Without loss of generality, we assume that $\tht_0(x_0)=1$ and $\rd_2 \tht_0 (x_0) >0$. 
	
	\medskip
	
	\noindent \underline{Step 1: Lagrangian flow map.} We recall that the flow map $\Phi(t,\cdot) : \overline{\bbR^2_+} \to \bbR^2_+$ is well-defined for any $0\le t \le \dlt$ as the unique solution of \begin{equation*}
		\begin{split}
			\frac{\ud}{\ud t}\Phi(t,x) = u(t,\Phi(t,x)) , \qquad \Phi(0,x)=x. 
		\end{split}
	\end{equation*}        Due to the uniform boundedness of $u$, we can see $\supp \tht(t) \subset B(0;1)$ for all $t \in [0,\delta]$. For the estimate for $\Phi_1$, we recall \eqref{eq:Phi-1-est}. Let $x \in \supp \tht_0$ with $x_2 < a$. By Lemma~\ref{lem:vel} we have		\begin{equation*}
		\begin{aligned}
			\frac {\ud}{\ud t}(\Phi_2(t,x_0) - \Phi_2(t,x)) &= u_2(t,\Phi(t,x_0)) - u_2(t,\Phi(t,x)) \\
			&\leq U_2(t,\Phi(t,x_0)) - U_2(t,\Phi(t,x)) + C \| \tht \|_{X^{\alpha}} |\Phi(t,x_0) - \Phi(t,x)|.
		\end{aligned}
	\end{equation*}		Note that \begin{equation*}
		\begin{gathered}
			U_2(t,\Phi(t,x_0)) - U_2(t,\Phi(t,x)) \\
			= -\int^{\Phi_1(t,x)}_{0} \partial_1 U_2(t,\tau,\Phi_2(t,x_0)) \,\ud \tau + \int_{\Phi_2(t,x)}^{\Phi_2(t,x_0)} \partial_2 U_2(t,\Phi_1(t,x),\tau) \,\ud \tau.
		\end{gathered}
	\end{equation*}      Using \eqref{eq:U2-est} and \eqref{eq:U2-est-unbounded}, we can see \begin{equation*}
		\begin{gathered}
			-\int^{\Phi_1(t,x)}_{0} \partial_1 U_2(t,\tau,\Phi_2(t,x_0)) \,\ud \tau \\
			= -\int^{\Phi_1(t,x)}_{0} (\partial_1 U_2(t,\tau,\Phi_2(t,x_0)) - C_{\alpha}\tau^{-\alpha}\tht_0(x_0)) \,\ud \tau - \int_{0}^{\Phi_1(t,x)} C_{\alpha}\tau^{-\alpha}\tht_0(x_0) \,\ud \tau \\
			\leq C\| \partial_2 \tht \|_{L^{\infty}} |\Phi_1(t,x) - \Phi_1(t,x_0)| - \frac{C_{\alpha}}{1-\alpha} \tht_0(x_0) \Phi_1(t,x)^{1-\alpha}
		\end{gathered}
	\end{equation*}     and      \begin{equation*}
		\begin{gathered}
			\int_{\Phi_2(t,x)}^{\Phi_2(t,x_0)} |\partial_2 U_2(t,\Phi_1(t,x),\tau)| \,\ud \tau \leq C \| \partial_2 \tht \|_{L^{\infty}} |\Phi_2(t,x) -\Phi_2(t,x_0)|.
		\end{gathered}
	\end{equation*} Thus, we have \begin{equation*}
		\begin{gathered}
			\frac {\ud}{\ud t}(\Phi_2(t,x_0) - \Phi_2(t,x)) \leq - \varepsilon \tht_0(x_0) x_1^{1-\alpha} + C\| \tht \|_{X^{\alpha}} |\Phi(t,x_0) - \Phi(t,x)|
		\end{gathered}
	\end{equation*} for some constant $\varepsilon>0$ that only depending on $\alpha$. We assume that the quantity $\Phi_2(t,x_0) - \Phi_2(t,x)$ decreases over time. Indeed, this can be proved by showing that the right-hand side is negative along with the continuity argument. Then, it holds $\frac{1}{C}x_1^2 \leq |\Phi(t,x_0) - \Phi(t,x)|^2 \leq C(x_1^2 + (x_2-a)^2)$. Here, we take $x$ with $a - x_2 = \ell^{-(1-\gamma)}$ and $x_1 = \ell^{-1}$ for any given $\ell > 1$ and $\gamma \in (0,\alpha)$. Then, we have for sufficiently large $\ell$ that   \begin{equation*}
		\begin{split}
			\frac {\ud}{\ud t}(\Phi_2(t,x_0) - \Phi_2(t,x)) &\leq -\varepsilon \ell^{-(1-\alpha)} + C \| \tht \|_{X^{\alpha}} \ell^{-(1-\gamma)} \leq -\frac{\varepsilon}{2} \ell^{-(1-\alpha)}.
		\end{split}
	\end{equation*}		Integrating over times gives    \begin{equation*}
		\begin{split}
			\Phi_2(t_\ell,x_0) - \Phi_2(t_\ell,x) \leq \ell^{-(1-\gamma)} -\frac{\varepsilon}{2} \ell^{-(1-\alpha)} t_\ell \leq 0, \qquad t_\ell := \frac {2}{\varepsilon} \ell^{-(\alpha-\gamma)}.
		\end{split}
	\end{equation*}		Note that $t_\ell \to 0$ as $\ell \to \infty$. Thus, there exists $t^* \in (0,t_\ell] \subset(0,\delta]$ such that $\Phi_2(t^*,x_0) = \Phi_2(t^*,x)$.
	
	\medskip
	
	\noindent \underline{Step 2: Norm inflation.} Now, we are ready to finish the proof of Theorem~\ref{thm:nonexist1}. For any $\beta \in (1-\alpha,1]$, it is clear that     \begin{equation*}
		\begin{split}
			\frac {\tht(t,\Phi(t^*,x_0)) - \tht(t^*,\Phi(t,x))}{|\Phi(t^*,x_0) - \Phi(t^*,x)|^{\beta}} = \frac {\tht_0(x_0) - \tht_0(x)}{|x_0-x|} |x_0 - x|^{1-\beta} \left( \frac {|x_0-x|}{|\Phi_1(t^*,x)|} \right)^{\beta}.
		\end{split}
	\end{equation*}		From \eqref{eq:Phi-1-est}, we have $$\left( \frac {|x_0-x|}{|\Phi_1(t^*,x)|} \right)^{\beta} \geq \frac{1}{C} \ell^{\beta\gamma}.$$ Since $|x_0-x|^2 = \frac {1+\ell^{2\gamma}}{\ell^{2\gamma}}|a-x_2|^2 = x_1^2(1+\ell^{2\gamma})$, we can see $$\frac {\tht_0(x_0) - \tht_0(x)}{|x_0-x|} = \frac {\tht_0(0,a) - \tht_0(0,x_2)}{|a-x_2|} \left(\frac {\ell^{2\gamma}}{1+\ell^{2\gamma}}\right)^{\frac {1}{2}} + \frac {\tht_0(0,x_2) - \tht_0(x)}{|x_1|} \left(\frac {1}{1+\ell^{2\gamma}}\right)^{\frac {1}{2}}.$$ Combining the above gives $$\frac {\tht(t^*,\Phi(t,x_0)) - \tht(t,\Phi(t^*,x))}{|\Phi(t^*,x_0) - \Phi(t^*,x)|^{\beta}} \geq \frac{1}{C} \frac {\tht_0(0,a) - \tht_0(0,x_2)}{|a-x_2|} \left(\frac {\ell^{2\gamma}}{1+\ell^{2\gamma}}\right)^{\frac {1}{2}} \ell^{\beta\gamma-(1-\beta)(1-\gamma)}$$ Note that $\beta\gamma-(1-\beta)(1-\gamma) > 0$ for $\gamma > 1-\beta$, where $1-\beta < \alpha$. Thus, taking $\gamma \in (1-\beta,\alpha)$ and passing $\ell \to \infty$, we obtain $\sup_{t \in [0,\delta]}\| \tht(t) \|_{C^{\beta}} = \infty$ for any $\delta > 0$. This finishes the proof.  
\end{proof}

\begin{proof}[Proof of Theorem \ref{thm:nonexist2}]
	We take the initial data $\tht_0$ satisfying the assumptions of Theorem \ref{thm:nonexist1}. Then, we have a boundary point $x_0 = (0,a)$ such that $\tht_0(x_0) \ne 0$, $\rd_2\tht_0(x_0) \ne 0$. For simplicity, we assume that $\tht_0(x_0)=1$ and $\rd_2 \tht_0 (x_0) >0$. Suppose that there exists a solution $\tht \in L^{\infty}([0,\delta];C^{\alpha})$ with $\sup_{t \in [0,\delta]} \| \tht(t) \|_{C^{\alpha}} \leq M$ for some $\delta>0$ and $M>0$. Without loss of generality, we take $\delta $ small enough to satisfy $\delta M \ll 1$. In the following, we obtain a contradiction to this assumption.
	
	\medskip
	
	\noindent \underline{Step 1: Lagrangian flow map.} Applying Lemma~\ref{lem:vel2} or Remark~\ref{rmk:SQG}, we have a flow map $\Phi(t,\cdot) : \bbR^2_+ \to \bbR^2_+$ which is well-defined for any $0\le t \le \dlt$ as the unique solution of \begin{equation*}
		\begin{split}
			\frac{\ud}{\ud t}\Phi(t,x) = u(t,\Phi(t,x)) , \qquad \Phi(0,x)=x. 
		\end{split}
	\end{equation*}      Note that $\Phi$ is not defined on the boundary set when $\alpha = 1$. Due to the uniform boundedness of $u$, we can see $\supp \tht(t) \subset B(0;1)$ for all $t \in [0,\delta]$. In the following, we show \eqref{eq:Phi-2-est} only for $\alpha \neq 1$. One can obtain the same inequality for $\alpha=1$ by the use of Remark~\ref{rmk:SQG}. From Lemma~\ref{lem:vel2}, $u_1$ is uniformly Log-Lipschitz, so together with $u_1(t,(0,x_2))=0$, the first component of the flow map satisfies the estimate \begin{equation*}
		\begin{split}
			\left| \frac{\ud}{\ud t} \Phi_1(t,x) \right| \le CM \Phi_1(t,x) \log(\frac{1}{|\Phi_1(t,x)|}),
		\end{split}
	\end{equation*} Thus, \begin{equation}\label{eq:Phi-1-est2}
		\begin{split}
			x_1 \exp(\exp(-CMt)) \le	\Phi_1(t,x) \le x_1 \exp(\exp(CMt)).
		\end{split}
	\end{equation}     Let $x,x' \in \bbR^2_{+}$ with $x'_1 < x_1$. By Lemma~\ref{lem:vel2} we have	\begin{equation*}
		\begin{gathered}
			\frac {\ud}{\ud t}(\Phi_2(t,x') - \Phi_2(t,x)) = u_2(t,\Phi(t,x')) - u_2(t,\Phi(t,x)) \\
			\leq \tht(\Phi(t,x'))(f(\Phi(t,x')) - f(\Phi(t,x)) + CM|\Phi(t,x) - \Phi(t,x')| \log\left(\frac {1}{|\Phi(t,x) - \Phi(t,x')|}\right).
		\end{gathered}
	\end{equation*}		Note that \begin{equation*}
		\begin{gathered}
			\tht(\Phi(t,x'))(f(\Phi(t,x')) - f(\Phi(t,x)) \\
			= \tht_0(x')(f(\Phi(t,x')) - f(\Phi_1(t,x),\Phi_2(t,x'))) + \tht_0(x')(f(\Phi_1(t,x),\Phi_2(t,x')) - f(\Phi(t,x))).
		\end{gathered}
	\end{equation*} Since \eqref{eq:U2-est2}, \eqref{eq:U2-est-unbounded2}, and \eqref{eq:Phi-1-est2} imply \begin{equation*}
		\begin{split}
			\tht_0(x')(f(\Phi(t,x')) - f(\Phi_1(t,x),\Phi_2(t,x'))) &= -\tht_0(x') \int_{\Phi_1(t,x')}^{\Phi_1(t,x)} \partial_1 f(\tau,\Phi_2(t,x')) \,\ud \tau \\
			&\leq -\frac{1}{C} \tht_0(x') \left(x_1^{1-\alpha} - {x'_1}^{1-\alpha}\right)
		\end{split}
	\end{equation*}		and  \begin{equation*}
		\begin{split}
			\left| \tht_0(x')(f(\Phi_1(t,x),\Phi_2(t,x')) - f(\Phi(t,x))) \right| \leq C\tht_0(x')\left|\Phi_2(t,x) - \Phi_2(t,x')\right|,
		\end{split}
	\end{equation*}		it follows \begin{equation}\label{eq:Phi-2-est}
		\begin{gathered}
			\frac {\ud}{\ud t}(\Phi_2(t,x') - \Phi_2(t,x)) \\
			\leq -\frac{1}{C}\tht_0(x') \left(x_1^{1-\alpha} - {x'_1}^{1-\alpha}\right) + CM|\Phi(t,x) - \Phi(t,x')| \log\left( \frac {1}{|\Phi(t,x) - \Phi(t,x')|}\right).
		\end{gathered}
	\end{equation}
	Here, we consider $x$ and $x'$ with $x_1=\ell^{-1}$, $x'_1 = \ell^{-2}$, $a-x_2 = \ell^{-(1-\gamma)}$, and $a-x'_2 = \ell^{-(2-\gamma)}$ for given $\ell>1$. As in the proof of Theorem~\ref{thm:nonexist1}, we can assume that the quantity $\Phi_2(t,x') - \Phi_2(t,x)$ decreases over time. Then, it holds by \eqref{eq:Phi-1-est2} that $\frac{1}{C}x_1^2 \leq |\Phi(t,x_0) - \Phi(t,x)|^2 \leq C(x_1^2 + (x_2-x'_2)^2)$. For sufficiently large $\ell$, we can see   \begin{equation*}
		\begin{split}
			-\frac{1}{C}\tht_0(x') \left(\Phi_1(t,x)^{1-\alpha} - \Phi_1(t,x')^{1-\alpha}\right) \leq -\varepsilon \tht_0(x') \ell^{-(1-\alpha)}
		\end{split}
	\end{equation*}		 for some arbitrary constant $\varepsilon>0$. Thus, we obtain      \begin{equation*}
		\begin{split}
			\frac {\ud}{\ud t}(\Phi_2(t,x') - \Phi_2(t,x)) \leq -\varepsilon \tht_0(x') \ell^{-(1-\alpha)} + CM\ell^{-(1-\gamma)} \log \ell \leq -\frac{\varepsilon}{2} \ell^{-(1-\alpha)}
		\end{split}
	\end{equation*}		for sufficiently large $\ell$, if $\gamma \in (0,\alpha)$. It is used in the last inequality that $x' \to x_0$ as $\ell \to \infty$. Then, it follows    \begin{equation*}
		\begin{split}
			\Phi_2(t_\ell,x') - \Phi_2(t_\ell,x) \leq \ell^{-(1-\gamma)} - \ell^{-(2-\gamma)} - \frac{\varepsilon}{2} \ell^{-(1-\alpha)} t_\ell \leq 0, \qquad t_\ell := \frac {2}{\varepsilon} \ell^{-(\alpha-\gamma)}.
		\end{split}
	\end{equation*}		Note that $t_\ell \to 0$ as $\ell \to \infty$. Thus, there exists $t^* \in (0,t_\ell] \subset(0,\delta]$ such that $\Phi_2(t^*,x_0) = \Phi_2(t^*,x)$.
	
	\medskip
	
	\noindent \underline{Step 2: Norm inflation.} Now, we are ready to finish the proof of Theorem~\ref{thm:nonexist2}. For any $\beta \in (1-\alpha,1]$, we have from \eqref{eq:Phi-1-est2} that    \begin{equation*}
		\begin{split}
			\frac {\tht(t,\Phi(t^*,x')) - \tht(t^*,\Phi(t,x))}{|\Phi(t^*,x') - \Phi(t^*,x)|^{\beta}} &= \frac {\tht_0(x') - \tht_0(x)}{|x'-x|} |x' - x|^{1-\beta} \left( \frac {|x'-x|}{|\Phi_1(t^*,x')-\Phi_1(t^*,x)|} \right)^{\beta} \\
			&\geq \frac{1}{C} \frac {\tht_0(x') - \tht_0(x)}{|x'-x|} \ell^{\beta\gamma-(1-\beta)(1-\gamma)}.
		\end{split}
	\end{equation*}		It is not hard to show that $$\frac {\tht_0(x') - \tht_0(x)}{|x'-x|} \geq \frac{1}{2} \partial_2 \tht_0(x_0)$$ for suffieicently large $\ell$. Thus, if we take $\gamma \in (1-\beta,\alpha)$, then $\beta\gamma-(1-\beta)(1-\gamma)>0$ gives that $$\sup_{t \in (0,\delta]}\| \tht(t) \|_{C^{\alpha}} \geq \frac {\tht(t,\Phi(t^*,x')) - \tht(t^*,\Phi(t,x))}{|\Phi(t^*,x') - \Phi(t^*,x)|^{\beta}} > M$$ for some $\ell$. This finishes the proof. 
\end{proof}

\section{Proof of Lemma~\ref{lem:vel2}}
\begin{proof} Let $\bar{\tht}$ be the odd extension of $\tht$ in $\bbR^2$; that is, we set $\bar{\tht}(x_1,x_2)=-\tht(-x_1,x_2)$ for $x_1<0$ and $\bar{\tht}=\tht$ otherwise. Recalling the Biot--Savart law \eqref{eq:Biot-Savart-aSQG}, we have for $x\in \bbR^2_{+}$ that	\begin{equation*}
		\begin{split}
			u_1(x) &=  -\int_{\bbR^2} \frac{x_2-y_2}{|x-y|^{2+\alp}} \bar{\tht}(y) \, \ud y.
		\end{split}
	\end{equation*} 	For $|x| > 2$, the assumption $\supp \tht \subset B(0;1)$ gives $|\nabla^\ell u_1(x)| \leq C \| \tht \|_{L^{\infty}}$ for all $\ell \in \bbN$. Let $x,x' \in B(0;4)$ and $\tilde{\tht}$ be the even extension of $\tht$ in $\bbR^2$. Then, we have from $\supp \tht \subset \supp \varphi$ that 	\begin{equation*}
		\begin{gathered}
			u_1(x) - u_1(x') =  -\int_{\bbR^2} \left( \frac{x_2-y_2}{|x-y|^{2+\alp}} - \frac{x'_2-y_2}{|x'-y|^{2+\alp}} \right) \left(\tilde{\tht}(y) - \tht(x)\right) \varphi(y) \, \ud y \\
			-2 \int_{\bbR^2_{-}} \left( \frac{x_2-y_2}{|x-y|^{2+\alp}} - \frac{x'_2-y_2}{|x'-y|^{2+\alp}} \right) \bar{\tht}(y) \varphi(y) \, \ud y- \tht(x)\int_{\bbR^2} \left( \frac{x_2-y_2}{|x-y|^{2+\alp}} - \frac{x'_2-y_2}{|x'-y|^{2+\alp}} \right) \varphi(y) \, \ud y.
		\end{gathered}
	\end{equation*}		By the change of variables, the third integral is bounded by 		\begin{equation*}
		\begin{gathered}
			|\tht(x)| \int_{\bbR^2} \frac{|y_2|}{|y|^{2+\alp}} \left| \varphi(x-y) - \varphi(x'-y) \right| \, \ud y \leq C \| \tht \|_{L^{\infty}} |x-x'|.
		\end{gathered}
	\end{equation*}		We calculate the first integral dividing the region into $B(x;2|x-x'|)$ and $\bbR^2 \setminus B(x;2|x-x'|)$. In the first case, we have 	\begin{equation*}
		\begin{gathered}
			\left| \int_{B(x;2|x-x'|)} |x-y|^{\alpha}\left( \frac{x_2-y_2}{|x-y|^{2+\alp}} - \frac{x'_2-y_2}{|x'-y|^{2+\alp}} \right) \frac{\tilde{\tht}(y) - \tht(x)}{|x-y|^{\alpha}} \varphi(y) \, \ud y \right| \leq C |x-x'| \| \tht \|_{C^{0,\alpha}}.
		\end{gathered}
	\end{equation*} 	On the other hand, since it holds for $y \in \bbR^2 \setminus B(x;2|x-x'|)$ that \begin{equation*}
		\begin{split}
			\left| |y-x|^{\alpha}\left( \frac{x_2-y_2}{|x-y|^{2+\alp}} - \frac{x'_2-y_2}{|x'-y|^{2+\alp}} \right) \right| &= |y-x|^{\alpha} \left| \int_0^1 \partial_{\tau} \left(\frac{\tau (x_2-y_2) + (1-\tau)(x'_2-y_2)}{|\tau (x-y) + (1-\tau)(x'-y)|^{2+\alp}} \right) \,\ud \tau \right| \\
			&\leq C|x-x'| \frac {1}{|x-y|^{2}},
		\end{split}
	\end{equation*} 	it follows 	\begin{equation*}
		\begin{gathered}
			\left| \int_{\bbR^2 \setminus B(x;2|x-x'|)} |y-x|^{\alpha}\left( \frac{x_2-y_2}{|x-y|^{2+\alp}} - \frac{x'_2-y_2}{|x'-y|^{2+\alp}} \right) \frac{\bar{\tht}(y) - \bar{\tht}(x)}{|y-x|^{\alpha}} \varphi(y) \, \ud y \right| \\
			\leq C |x-x'| \log (10 + \frac{1}{|x-x'|}) \| \tht \|_{C^{0,\alpha}}.
		\end{gathered}
	\end{equation*} 	To estimate the remainder integral, we rewrite it as 	\begin{equation*}
		\begin{gathered}
			-2 \int_{\bbR^2_{-}} \left( \frac{x_2-y_2}{|x-y|^{2+\alp}} - \frac{x'_2-y_2}{|x'-y|^{2+\alp}} \right) \left(-\tht(-y_1,y_2) + \tht(x) \right) \varphi(y) \, \ud y \\
			+2 \tht(x)\int_{\bbR^2_{-}} \left( \frac{x_2-y_2}{|x-y|^{2+\alp}} - \frac{x'_2-y_2}{|x'-y|^{2+\alp}} \right) \varphi(y) \, \ud y.
		\end{gathered}
	\end{equation*}		We can treat the first integral as before to be bounded by $C |x-x'| \log (10 + \frac{1}{|x-x'|}) \| \tht \|_{C^{0,\alpha}}.$ On the other hand, we have by integration by parts that 	\begin{equation*}
		\begin{split}
			2 \tht(x)\int_{\bbR^2_{-}} \left( \frac{x_2-y_2}{|x-y|^{2+\alp}} - \frac{x'_2-y_2}{|x'-y|^{2+\alp}} \right) \varphi(y) \, \ud y &= \frac {2}{\alpha} \tht(x)\int_{\bbR^2_{-}} \left( \frac{1}{|x-y|^{\alp}} - \frac{1}{|x'-y|^{\alp}} \right) \partial_2 \varphi(y) \, \ud y \\
			&\leq C |x-x'| \| \tht \|_{L^{\infty}}.
		\end{split}
	\end{equation*}		Thus, we obtain \eqref{eq:u-logLip2} for $u_1$ by combining the above inequalities. Now, we estimate	\begin{equation*}
		\begin{split}
			u_2(x) &=  \int_{\bbR^2} \frac{x_1-y_1}{|x-y|^{2+\alp}} \bar{\tht}(y) \, \ud y.
		\end{split}
	\end{equation*}		Due to $\partial_1 u_1 = -\partial_2 u_2$, it holds $\| \partial_2 u_2 \|_{L^{\infty}} \leq C \| \tht \|_{C^{0,\alpha}}$. From $\supp \tht \subset \supp \varphi$, we have 	\begin{equation*}
		\begin{gathered}
			u_2(x) - u_2(x') =  \int_{\bbR^2} \left( \frac{x_1-y_1}{|x-y|^{2+\alp}} - \frac{x'_1-y_1}{|x'-y|^{2+\alp}} \right) \left(\tilde{\tht}(y) - \tht(x)\right) \varphi(y) \, \ud y \\
			+2 \int_{\bbR^2_{-}} \left( \frac{x_1-y_1}{|x-y|^{2+\alp}} - \frac{x'_1-y_1}{|x'-y|^{2+\alp}} \right) \bar{\tht}(y) \varphi(y) \, \ud y+ \tht(x)\int_{\bbR^2} \left( \frac{x_1-y_1}{|x-y|^{2+\alp}} - \frac{x'_1-y_1}{|x'-y|^{2+\alp}} \right) \varphi(y) \, \ud y.
		\end{gathered}
	\end{equation*}		It is clear that the first and third integral is bounded by $C|x-x'| \| \tht \|_{C^{0,\alpha}}$. We can see that the remainder term equals to 	\begin{equation*}
		\begin{gathered}
			2 \int_{\bbR^2_{-}} \left( \frac{x_1-y_1}{|x-y|^{2+\alp}} - \frac{x'_1-y_1}{|x'-y|^{2+\alp}} \right) \left(-\tht(-y_1,y_2) + \tht(x) \right) \varphi(y) \, \ud y \\
			-2 \tht(x)\int_{\bbR^2_{-}} \left( \frac{x_1-y_1}{|x-y|^{2+\alp}} - \frac{x'_1-y_1}{|x'-y|^{2+\alp}} \right) \varphi(y) \, \ud y.
		\end{gathered}
	\end{equation*}		Since we can estimate the first integral similarly, we skip it. Integration by parts gives 	\begin{equation*}
		\begin{gathered}
			-2 \tht(x)\int_{\bbR^2_{-}} \left( \frac{x_1-y_1}{|x-y|^{2+\alp}} - \frac{x'_1-y_1}{|x'-y|^{2+\alp}} \right) \varphi(y) \, \ud y \\
			= -\frac {2}{\alpha} \tht(x) \left( \int_{\bbR} \frac {1}{|x-(0,y_2)|^{\alpha}} \varphi(0,y_2) \,\ud y_2 - \int_{\bbR} \frac {1}{|x'-(0,y_2)|^{\alpha}} \varphi(0,y_2) \,\ud y_2 \right) \\
			-\frac {2}{\alpha} \tht(x)\int_{\bbR^2_{-}} \left( \frac{1}{|x-y|^{\alp}} - \frac{1}{|x'-y|^{\alp}} \right) \partial_1 \varphi(y) \, \ud y .
		\end{gathered}
	\end{equation*}        Since we can have by \eqref{eq:U2-def2} that 	\begin{equation*}
		\begin{gathered}
			\left| -2 \tht(x)\int_{\bbR^2_{-}} \left( \frac{x_1-y_1}{|x-y|^{2+\alp}} - \frac{x'_1-y_1}{|x'-y|^{2+\alp}} \right) \varphi(y) \, \ud y - \tht(x)(f(x) - f(x')) \right| \\
			\leq \left| -\frac {2}{\alpha} \tht(x)\int_{\bbR^2_{-}} \left( \frac{1}{|x-y|^{\alp}} - \frac{1}{|x'-y|^{\alp}} \right) \partial_1 \varphi(y) \, \ud y \right| \leq C |x-x'| \| \tht \|_{L^{\infty}},
		\end{gathered}
	\end{equation*}        \eqref{eq:u-logLip2} is obtained. Integration by parts gives \eqref{eq:U2-est2}. Note that			\begin{equation*}
		\begin{split}
			f(x) - f(x'_1,x_2) &= -\frac {2}{\alpha}\int_{\bbR} \left( \frac {1}{|x-(0,y_2)|^{\alpha}} - \frac {1}{|(x'_1,x_2)-(0,y_2)|^{\alpha}} \right)\varphi(0,y_2) \,\ud y_2 \\
			&= -\frac {2}{\alpha} \int_0^1 \int_{\bbR} \partial_{\tau} \left(\frac {1}{|(x'_1 + \tau (x_1-x'_1),x_2) - (0,y_2)|^{\alpha}} \right) \varphi(0,y_2) \,\ud y_2 \ud \tau \\
			&= 2\int_0^1 \int_{\bbR} \frac {(x_1-x'_1)(x'_1+\tau(x_1-x'_1))}{|(x'_1 + \tau (x_1-x'_1),x_2) - (0,y_2)|^{2+\alpha}} \varphi(0,y_2) \,\ud y_2 \ud \tau.
		\end{split}
	\end{equation*}		By the change of variables $y_2=x_2 + (x'_1 + \tau (x_1-x'_1))z$, we have	\begin{equation*}
		\begin{gathered}
			\frac{f(x) - f(x'_1,x_2)}{x_1-x'_1} \\
			= 2\int_0^1 \frac {1}{(x'_1 + \tau (x_1-x'_1))^{\alpha}} \int_{\bbR} \frac {1}{(1+z^2)^{1+\alpha/2}} \varphi(0,x_2 + (x'_1 + \tau (x_1-x'_1))z) \,\ud z \ud \tau.
		\end{gathered}
	\end{equation*}		Passing $x'_1$ to $x_1$, we obtain \eqref{eq:U2-est-unbounded2}. This finishes the proof. 
\end{proof}

\bibliographystyle{amsplain}

\end{document}